\documentclass[10pt]{article}
\usepackage{color}
\usepackage{amssymb}
\usepackage{amsthm,array,amssymb,amscd,amsfonts,latexsym, url}
\usepackage{amsmath}
\usepackage[all]{xy}
\newtheorem{theo}{Theorem}[section]
\newtheorem{prop}[theo]{Proposition}
\newtheorem{claim}[theo]{Claim}
\newtheorem{lemm}[theo]{Lemma}
\newtheorem{coro}[theo]{Corollary}
\newtheorem{rema}[theo]{Remark}
\newtheorem{Defi}[theo]{Definition}

\newtheorem{conj}[theo]{Conjecture}

\newtheorem{question}[theo]{Question}
\newtheorem{example}[theo]{Example}

\title{Chow rings  and gonality of general abelian varieties}
\author{Claire Voisin
\\Coll\`ege de France}

\date{}

\newfont{\gothic}{eufb10}
\begin{document}
\maketitle
\setcounter{section}{-1}

\begin{abstract} We study the (covering) gonality of  abelian varieties and their orbits of zero-cycles for rational equivalence. We show that any orbit for rational equivalence of zero-cycles of degree $k$ has dimension at most $k-1$. Building on the work of Pirola, we show that
very general abelian varieties of dimension $g$ have covering gonality $k\geq f(g)$   where $f(g)$ grows like  ${\rm log}\,g$. This answers   a question asked by Bastianelli,  De Poi,  Ein,  Lazarsfeld and B. Ullery. We also obtain results on the Chow ring of very general abelian varieties, eg. if $g\geq 2k-1$, for any divisor $D\in {\rm Pic}^0(A)$,  $D^k$ is not a torsion cycle.
\end{abstract}
\section{Introduction}
The gonality of a projective variety $X$  is defined in this paper as the minimal gonality of the normalization
of an irreducible  curve $C\subset X$. In the case of an abelian variety, the gonality is the same as the covering gonality studied in \cite{lazarsfeld}.
One of the main results  of this paper  answers affirmatively a question asked in \cite{lazarsfeld} concerning the gonality of a very general abelian variety $A$, namely, whether it grows to infinity with $g={\rm dim}\,A$.
\begin{theo} \label{theogon} Let $A$ be a very general abelian variety of dimension $g$. Then if $g\geq 2^{k-2}(2k-1)+(2^{k-2}-1)(k-2)$, the gonality of $A$ is at least $k+1$.
\end{theo}
In other words, a very general abelian variety of dimension $\geq 2^{k-2}(2k-1)+(2^{k-2}-1)(k-2)$ does not contain a curve
 of gonality $\leq k$.
  This theorem is presumably not optimal. What seems reasonable is the following bound:
  \begin{conj} \label{conj} Let $A$ be a very general abelian variety of dimension $g$. Then
  if $g\geq 2k-1$, the gonality of $A$ is at least $k+1$.
  \end{conj}
  We will discuss in Section \ref{secdiscussion} a strategy  towards proving this statement and some evidence  for it.
  Theorem \ref{theogon}  generalizes the following result by Pirola \cite{pirola}:
  \begin{theo}\label{theopirola}(Pirola) A very general abelian variety
of dimension at least $3$ does not contain a hyperelliptic curve.
  \end{theo}
   We will  in fact use in the proof of
Theorem \ref{theogon} (and also  \ref{coroorbit}, \ref{theochowring} below)  some of the arguments in \cite{pirola}
that we generalize in Section \ref{secpirola}.
Theorem \ref{theogon} will be obtained as a consequence of the study of
$0$-cycles modulo rational equivalence on abelien varieties. This generalized setting already appears in the paper
\cite{alzatipirola} where some improvements of Theorem \ref{theopirola} (for example on  the non-existence of trigonal curves on very general abelian varieties of dimension $\geq 4$) were obtained.
  In this paper, Chow groups  with $\mathbb{Q}$-coefficients of a variety $X$ are denoted ${\rm CH}(X)$.
  Rational equivalence of $0$-cycles is not very well understood, despite Mumford's theorem \cite{mumford}. The most striking phenomenon is the existence of surfaces (eg. of Godeaux type, see  \cite{voisingodeaux}) which are of general type but have trivial ${\rm CH}_0$-group.
In the papers \cite{voisink3}, \cite{voisinisot}, we emphasized nevertheless the geometric importance
 of the study of   orbits $$\mid Z\mid=\{Z'\in  X^{(k)},\,\,Z'\,\,{\rm rationally\,\, equivalent \,\,to }\,\,Z\,\,{\rm in}\,\,X\}$$ of  degree $k$  $0$-cycles $Z$ of  $X$ under rational equivalence,  particularly in the case of $K3$ surfaces and hyper-K\"ahler manifolds. It is a general fact that these orbits are countable unions of closed algebraic subsets in the symmetric product of the considered variety, so that their dimension is well-defined.
 Below we denote by  $\{x\}$ the  $0$-cycle of a point $x\in A$ and $0_A$ will be  the origin of $A$. The following results  concerning orbits $|Z|\subset A^{(k)}$ for rational equivalence, and in particular
the orbit $|k\{0_A\}|$, can be regarded as a  Chow-theoretic version of Theorem \ref{theogon}.
  \begin{theo} \label{coroorbit}
  (i) For any abelian variety $A$ and integer $k\geq 1$,   any orbit
 $|Z|\subset A^{(k)}$ has dimension $\leq k-1$.

 (ii)  If $A$ is very general of dimension
   $g\geq 2^{k}(2k-1)+(2^{k}-1)(k-2)$, $A$ has no positive dimensional orbit $|Z|$, with ${\rm deg}\,Z\leq k$.

   (iii) If $k\geq 2$ and  $A$ is very general of dimension
   $g\geq 2^{k-2}(2k-1)+(2^{k-2}-1)(k-2)$, $A$ has no positive dimensional orbit of the form $|Z'+2\{0_A\}|$, with $Z'$ effective and ${\rm deg}\,Z'\leq k-2$.

   (iv) If  $A$ is  a very general abelian variety of dimension $g\geq 2k-1$, the orbit
 $|k\{0_A\}|$ is countable.
\end{theo}
In fact, Theorem \ref{coroorbit}, (iii)  implies Theorem \ref{theogon}, because a $k$-gonal curve
$C\subset A$,  with normalization $j: \widetilde{C}\rightarrow A$ and  divisor  $D\in {\rm Pic}^k\,\widetilde{C}$ with $h^0(\widetilde{C},D)\geq 2$ provides a positive dimensional orbit $\{j_*D'\}_{D'\in |D|}$ in $A^{(k)}$. We can assume one Weierstrass point $c\in \widetilde{C}$ of $|D|$, that is, a point $c$ such that $h^0(\widetilde{C},D(-2c))\not=0$, is mapped
to $0_A$ by $j$, which provides a positive dimensional orbit of the form $|Z'+2\{0_A\}|$, with $Z'$ effective and ${\rm deg}\,Z'\leq k-2$.

Item (i) of  Theorem \ref{coroorbit} will be proved in Section \ref{secdimorbit} (cf. Theorem  \ref{ledimorbit}).
 The estimates in Theorems \ref{theogon} and  \ref{coroorbit}, (ii) can probably be strongly improved. Estimate
 (i) in Theorem \ref{coroorbit} cannot be  improved. To start with, it is optimal for $g=1$ because
 for any degree $k$ divisor $D$ on an elliptic curve $E$ we have $|D|=\mathbb{P}^{k-1}\subset E^{(k)}$. This immediately implies that the statement is optimal for any $g$ because for
 abelian varieties $A=E\times B$ admitting an elliptic factor, we have $E^{(k)}\subset A^{(k)}$.
  In the case
  where    $g=2$,  we observe that orbits $|Z|\subset A^{(k)}$ are contained in the generalized Kummer
  variety $K_{k-1}(A)$ constructed by Beauville \cite{Beau}. (More precisely, this is true for the open set  of $|Z|$ parameterizing cycles where all points appear with multiplicity $1$ but this is secondary, cf. \cite{voisink3} for a discussion of cycles with multiplicities.) This variety is of dimension $2k-2$ and  has an everywhere nondegenerate holomorphic $2$-form for which any orbit $|Z|$ is totally isotropic, which implies the estimate (i) in the case $g=2$. Furthermore they are
  also orbits for rational equivalence in $K_{k-1}(A)$, as proved in \cite{marian}, hence they are as well constant cycles subvarieties in $K_{k-1}(A)$ in the sense of Huybrechts\cite {huy}. The question whether  Lagrangian (that is maximal dimension) constant cycles subvarieties
  exist in hyper-K\"ahler manifolds  is posed in \cite{voisinisot}.
  For a general abelian variety $A$,   choosing   a smooth curve $C\subset A$ of genus $g'$, we have
 $C^{(k)}\subset A^{(k)}$ for any $k$ and $C^{(k)}$ contains linear systems
 $\mathbb{P}^{k-g'}$, for $k\geq g'$. So when $k$ tends to infinity, the  estimate (i) has optimal growth in $k$.

  Theorem \ref{coroorbit}, (iv), which will be proved in Section \ref{secproofcoriv}, has the following immediate consequence (which is a much better estimate than the one given in Theorem \ref{theogon}):
\begin{coro} If $A$ is a very general abelian variety of dimension $g\geq 2k-1$, and $C\subset A$ is any curve with normalization
$\widetilde{C}$,  one has $h^{0}(C,\mathcal{O}_{\widetilde{C}}(kc))=1$  for any point $c\in \widetilde{C}$.
\end{coro}
This corollary could be regarded as the right generalization of Theorem \ref{theopirola}.
\begin{rema}{\rm   Pirola  proves in \cite{pirola2} that for a very general abelian variety
$A$
of dimension $g\geq4$, any curve $C\subset A$ has genus $\geq \frac{g(g-1)}{2}+1$.
This suggests that Theorem \ref{coroorbit}, (iv) is neither optimal, and that an inequality
$g\geq O(\sqrt{k})$ should already imply the countability of $|k\{0_A\}|$.}
\end{rema}
 We will give two proofs of Theorem \ref{coroorbit}, (iv). One of them will use Theorem \ref{theochowring} and Proposition \ref{theorelation}, which are statements of
 independent interest concerning  the Chow ring (as opposed to the Chow groups) of an abelian variety $A$, that we now describe. Here the Chow ring is relative to the  intersection product but one can also consider the ring structure given by the Pontryagin product $*$ defined by
$$z*z'=\mu_*(z\times z')$$
where $\mu:A\times A\rightarrow A$ is the sum map and $z\times z'=pr_1^*z\cdot pr_2^*z'$ for
$z,\,z'\in {\rm CH}(A)$. The two rings are related via the Fourier transform, see \cite{beauvillefourier}.
 Define \begin{eqnarray}
 \label{eqAk}
 A_k\subset A\end{eqnarray} to be the set of points $x\in A$
such that $(\{x\}-\{0_A\})^{*k}=0$ in ${\rm CH}_0(A)$. We can also define $\widehat{A}_k\subset \widehat{A}$
to be the set of $D\in {\rm Pic}^0(A)=:\widehat{A}$ such that
$D^k=0$ in ${\rm CH}^k(A)$. These two sets are in fact related as follows: choose a polarization $\theta$ on $A$, that is an ample divisor. The polarization gives an isogeny of abelian varieties
$$A\rightarrow \widehat{A},\,x\mapsto D_x:=\theta_x-\theta.$$

\begin{lemm}\label{lefourier} One has $D_x^k=0$ in ${\rm CH}^k(A)$ if and only if $(\{x\}-\{0_A\})^{*k}=0$ in ${\rm CH}_0(A)$.
\end{lemm}
\begin{proof} This follows from Beauville's formulas in
\cite[Proposition 6]{beauvillefourier}.  We get in particular,  the following equality:
\begin{eqnarray}
\label{eqbeau}
 \frac{\theta^{g-k}}{(g-k)!}D_x^k=\frac{\theta^g}{g!}*\gamma(x)^{*k},
 \end{eqnarray}
where $$\gamma(x):=\{0_A\}-\{x\}+\frac{1}{2}(\{0_A\}-\{x\})^{*2}+\ldots+\frac{1}{g}(\{0_A\}-\{x\})^{*g}=-{\rm log }(\{x\})\in {\rm CH}_0(A).$$ Here the logarithm is taken with respect to the
Pontryagin product $*$ and the development is finite because $0$-cycles of degree $0$ are nilpotent for the Pontryagin
product.
If $(\{0_A\}-\{x\})^{*k}=0$, then $\gamma(x)^{*k}=0$ and thus $D_x^k=0$ by (\ref{eqbeau}).
Conversely, if $D_x^k=0$, then $\gamma(x)^{*k}=0$ by (\ref{eqbeau}). But then also $(\{0_A\}-\{x\})^{*k}=0$ because
$\{x\}={\rm exp}(-\gamma(x))$. (Again ${\rm exp}(-\gamma(x))$ is a polynomial in $\gamma(x)$, hence well-defined, since $\gamma(x)$ is nilpotent for the $*$-product, see \cite{bloch}).
\end{proof}

\begin{theo}\label{theochowring} Let $A$ be an abelian variety of dimension $g$. Then

(i) ${\rm dim}\,A_k\leq k-1$, ${\rm dim}\,\widehat{A}_k\leq k-1$.

(ii) If $A$ is very general and  $g\geq 2k-1$, the sets $\widehat{A}_k$ and  $A_k$  are countable.
\end{theo}
Note that in both (i) and (ii), the two statements are equivalent by Lemma \ref{lefourier}, using the fact that
$A\mapsto \widehat{A}$ is an open map between moduli spaces, so that, if  $A$ is very general, so is
$\widehat{A}$.

The fact that Theorem \ref{theochowring} implies Theorem \ref{coroorbit}, (iv),  uses
the  following intriguing result  that does not seem to be written
anywhere, although some related results are available, in particular the results of \cite{colombo}, \cite{herbaut}, \cite{voisindiag}.

\begin{prop}\label{theorelation} Let $A$ be an abelian variety  and let $x_1,\ldots,\,x_k$ be $k$ points of $A$ such that $\sum_{i=1}^k\{x_i\}-k\{0_A\}=0$ in ${\rm CH}_0(A)$. Then for any $i=1,\ldots,\,k$
\begin{eqnarray}\label{eqrod} (\{x_i\}-\{0\})^{*k}=0\,\,{\rm in}\,\,{\rm CH}_0(A).
\end{eqnarray}
In other words, $x_i\in A_k$.
\end{prop}

 For the proof of  Theorem \ref{theochowring}, we will show how
  the dimension estimate provided by (i) implies the non-existence theorem stated in (ii). This is  obtained by establishing and applying Theorem \ref{thegenpirola},  that we will present in Section \ref{secpirola}. This theorem, which is obtained by a  direct generalization of  Pirola's arguments in \cite{pirola}, says that ``naturally defined subsets'' of abelian varieties
  (see Definition \ref{definat}), assuming they are proper subsets for abelian varieties of a  given  dimension $g$, are at most countable for very general abelian varieties of  dimension $\geq 2g-1$.

\vspace{1cm}

{\bf Thanks.} {\it This paper is deeply influenced by the reading of the beautiful Pirola paper
\cite{pirola}. I thank the organizers of the Barcelona  Workshop on Complex Algebraic Geometry dedicated to Pirola's 60th birthday for giving me the opportunity to speak about Pirola's work, which led  me to   thinking  to related
  questions.}
\section{\label{secpirola}Naturally defined subsets of abelian varieties}
The proof of Theorem \ref{theopirola} by Pirola  has two steps. First of all, Pirola shows that hyperelliptic curves in an abelian variety $A$, one of whose Weierstrass points coincides  with $0_A$, are rigid.
Secondly he deduces from this rigidity statement the nonexistence of any hyperelliptic curve in a very general abelian variety of dimension $\geq 3$ by an argument of
 specialization to abelian varieties isogenous to a product
 $B\times E$, that we now extend to cover
more situations.

\begin{Defi} \label{definat}We will say that a subset $\Sigma_A\subset A$  is natural if it satisfies
the following conditions:

(0)  $\Sigma_A\subset A$ is defined for any abelian variety $A$ and   is a countable union of closed algebraic subsets of $A$.

(i)  For any morphism
$f:A\rightarrow B$ of abelian varieties, $f(\Sigma_A)\subset \Sigma_B$.

(ii) For any family $\mathcal{A}\rightarrow S$, there is a countable union
of  closed algebraic subsets
$\Sigma_{\mathcal{A}}\subset \mathcal{A}$ such that the set-theoretic fibers satisfy
$\Sigma_{\mathcal{A},b}=\Sigma_{\mathcal{A}_b}$.
\end{Defi}

Recall that the dimension of a countable  union of closed algebraic subsets is defined as the supremum of the dimensions of its components (which are well defined since we are over
 the uncountable field $ \mathbb{C}$).
 \begin{rema}{\rm  By morphism of abelian varieties $A,\,B$, we mean group morphisms, that is, mapping $0_A$ to $0_B$.
 }
 \end{rema}
\begin{theo} \label{thegenpirola}  Let $\Sigma_A\subset A$ be a naturally defined subset.

(i) Assume that
for ${\rm dim}\,A=g_0$, one has $\Sigma_A\not=A$. Then for very general $A$ of dimension $\geq 2g_0-1$,
$\Sigma_A$ is at most countable.

(ii) Assume that ${\rm dim}\,\Sigma_A\leq k$ for any $A$. Then for very general $A$ of dimension $\geq 2k+1$,
$\Sigma_A$ is at most countable.

(iii) Assume that ${\rm dim}\,\Sigma_A\leq k-1$ for a very general abelian variety $A$ of dimension
$g_0\geq k$. Then for a very general abelian variety $A$ of dimension $\geq g_0+k-1$,
$\Sigma_A$ is at most countable.
\end{theo}
 Statement (ii) is a particular case of (i) where we do $g_0=k+1$. Both (i)  and (iii) will  follow from the following result:
 \begin{prop}\label{propourtheogenpirola}(a)  If for a very general abelian variety $B$  of dimension $g> k$, one has
 $$0<{\rm dim}\,\Sigma_B\leq k$$  then, for a very general abelian variety $A$  of dimension $g+1$, one has
 $${\rm dim}\,\Sigma_A\leq k-1.$$
 (b) If for a very general abelian variety $B$  of dimension $g>0$, $\Sigma_B$ is countable, then for $A$ very general of dimension $\geq g$, $\Sigma_A$ is countable.
 \end{prop}
 Indeed, applying Proposition \ref{propourtheogenpirola}, (a),  we conclude  in case (i) that
 the dimension of $\Sigma_A$ is strictly  decreasing  with $g\geq g_0$ as long as it is not equal to $0$, and by assumption it is not greater than $g_0-1$ for $g=g_0$. Hence  the dimension
 of $\Sigma_A$ must be
 $0$ for some $g\leq 2g_0-1$. By Proposition \ref{propourtheogenpirola}, (b), we then conclude that $\Sigma_A$ is countable for any  $g\geq 2g_0-1$.

  For case (iii), the argument is the same except that we start with
 dimension $g_0=k+1$ and we conclude similarly that the dimension of $\Sigma_A$ is strictly  decreasing  with $g\geq g_0$ as long as it is not equal to $0$. Furthermore, for $g=g_0$, this dimension  is equal to $k-1$. Hence  the dimension
 of $\Sigma_A$ must be
 $0$ for some $g\leq g_0+k-1$ and thus, by Proposition \ref{propourtheogenpirola}, (b),  $\Sigma_A$ is countable for any  $g\geq g_0+k-1$.  This proves Theorem \ref{thegenpirola} assuming Proposition
 \ref{propourtheogenpirola} that we now prove along the same lines as in \cite{pirola}.

 \begin{proof}[Proof of Proposition \ref{propourtheogenpirola}] Assume that ${\rm dim}\,\Sigma= k'$ for a very general abelian variety $A$ of dimension $g+1$. From the definition of a naturally defined subset, and by standard arguments involving the  properness and countability properties
 of relative Chow varieties, there exists, for each universal
 family $\mathcal{A}\rightarrow S$ of polarized abelian varieties with given polarization type
 $\theta$,  a family $\Sigma'_{\mathcal{A}}\subset \Sigma_{\mathcal{A}_{S'}}\subset\mathcal{A}_{S'}$, where $S'\rightarrow S$ is a generically finite dominant base-change morphism, $\mathcal{A}_{S'}\rightarrow S'$ is the base-changed family,  and the morphism  $\Sigma'_{\mathcal{A}}\rightarrow S'$ is flat, with irreducible fibers  of relative dimension $k'$. In other words, we choose one $k'$-dimensional component of $\Sigma_A$ for each $A$, and we can do this in families, maybe after passing to a generically finite cover of a Zariski open set of the base.

The main observation is the fact that there is a dense contable union of
 algebraic subsets $S'_\lambda\subset S'$ along which the fiber $\mathcal{A}_b$
 is isogenous to a product
 $B_\lambda\times E$ where $B$ is a generic abelian variety of dimension $g$ with polarization of type determined by $\lambda$ and $E$ is an elliptic curve ($\lambda$ also encodes the structure of the isogeny).
 Along each $S'_\lambda$, using axiom (i) of Definition
 \ref{definat}, possibly after passing to a generically finite cover $S''_\lambda$,
 we have a morphism
 $$p_\lambda:\mathcal{A}_{S''_\lambda}\rightarrow \mathcal{B}_{S''_\lambda}$$
 and $p_\lambda(\Sigma'_{\mathcal{A}_{S''_\lambda}})\subset \Sigma_{\mathcal{B}_{S''_\lambda}}$ by axiom (i) of Definition
 \ref{definat}.
 \begin{lemm} \label{lefini}If $\Sigma_B\subset B$ is a proper subset for a very general
  abelian variety $B$ of dimension $g$,  the morphism $p_{\lambda,\Sigma}:=p_{\lambda\mid \Sigma'_{\mathcal{A}_b}}:\Sigma'_{\mathcal{A}_b}\rightarrow \mathcal{B}_b$ is generically
 finite on its image for any point $b$ of $S''_\lambda$.
 \end{lemm}
 \begin{proof} As $p_\lambda(\Sigma'_{\mathcal{A}_b})\subset \Sigma_{\mathcal{B}_b}$ for $b\in S''_\lambda$, and we know by assumption that ${\rm dim}\,\Sigma_{\mathcal{B}_b}<g$, we conclude that $\Sigma'_{\mathcal{A}_b}\subset \mathcal{A}_b$ is a proper algebraic subset for $b\in S''_\lambda$, hence also for general $b\in S'$. For very general $b\in S'$, the cycle class
 $[\Sigma'_{\mathcal{A}_b}]\in H^{2l}(\mathcal{A}_b,\mathbb{Q})$, $l:={\rm codim }\,\Sigma'_{\mathcal{A}_b}$,  is a nonzero multiple of $\theta^l$ because the latter generates
  the space of degree $2l$  Hodge classes of a very general abelian  variety with polarizing class $\theta$.
 We thus conclude that $p_{\lambda*}([\Sigma'_{\mathcal{A}_b}])$ is nonzero in $H^{2l-2}(\mathcal{B}_b,\mathbb{Q})$, and as $\Sigma'_{\mathcal{A}_b}$ is irreducible by construction, it follows that $p_{\lambda,\Sigma}$ is generically finite on its image.
 \end{proof}
 Lemma \ref{lefini} applied to the case where
 $\Sigma_{\mathcal{B}_b}\subset \mathcal{B}_b$ is countable while ${\rm dim}\,\mathcal{B}_b>0$ implies statement (b) of Proposition \ref{propourtheogenpirola}.  We now concentrate on statement (a) and thus assume that
 ${\rm dim}\,B>{\rm dim}\,\Sigma_B>0$ for general $B$ of dimension $g$.
 We want to show that ${\rm dim}\,\Sigma'_{\mathcal{A}_b}<{\rm dim}\,\Sigma_{\mathcal{B}_b}$
 which, using  Lemma \ref{lefini} when $b\in S''_\lambda$ for some $\lambda$,  means   that $p_\lambda(\Sigma'_{\mathcal{A}_b})$ is not a component of
 $\Sigma_{\mathcal{B}_b}$.
 \begin{lemm}  In the situation above, the set  of varieties (of dimension $k'={\rm dim}\,\Sigma'_{\mathcal{A}_b}$) $\Sigma'_{\mathcal{A}_b,p_\lambda}:=p_\lambda(\Sigma'_{\mathcal{A}_b})$  and morphisms
$ p_{\lambda,\Sigma}: \Sigma'_{\mathcal{A}_b}\rightarrow \Sigma'_{\mathcal{A}_b,p_\lambda}$ for all $\lambda$'s is bounded up to birational transformations.
   \end{lemm}
   \begin{proof} Recall that $\Sigma'_{\mathcal{A}_b,p_\lambda}\subset \mathcal{B}_b$ is a
    proper subvariety of a very general abelian variety of dimension $g$
   with polarization of certain type,
   and $\Sigma'_{\mathcal{A}_b}\subset \mathcal{A}_b$ is the specialization of a subvariety (of codimension at least $2$ by Lemma \ref{lefini}) of a general abelian
   variety of dimension $g+1$ at a point $b$ which is Zariski dense in $S$. In both cases, it follows that the Gauss maps $g_A$ of $\Sigma'_{\mathcal{A}_b}\subset \mathcal{A}_b$ and $g_B$ of $\Sigma'_{\mathcal{A}_b,p_\lambda}\subset \mathcal{B}_b$, which take respective values
    in $G(k',g+1)={\rm Grass}(k',T_{\mathcal{A}_b,0_{\mathcal{A}_b}})$
    and $G(k',g)={\rm Grass}(k',T_{\mathcal{B}_b,0_{\mathcal{B}_b}})$,
     are generically finite on their images.
   We have the commutative diagram

   \begin{eqnarray}\label{numerodiag}
 \label{diagram} \xymatrix{
&\Sigma'_{\mathcal{A}_b}\ar[r]^{g_A}\ar[d]^{p_{\lambda,\Sigma}}& G(k',g+1)\ar[d]^{\pi_\lambda}&\\
&\Sigma'_{\mathcal{A}_b,p_\lambda}\ar[r]^{g_B}& G(k',g)& },
\end{eqnarray}
   where all the maps are rational maps and the rational map $\pi_\lambda:G(k',g+1)\dashrightarrow G(k',g)$ is induced by the
   linear map $dp_\lambda: T_{\mathcal{A}_b,0_{\mathcal{A}_b}}\rightarrow T_{\mathcal{B}_b,0_{\mathcal{B}_b}}$ which is also the quotient map $T_{\mathcal{A}_b,0_{\mathcal{A}_b}}\rightarrow T_{\mathcal{A}_b,0_{\mathcal{A}_b}}/T_{E,0}$. We observe here that the density of the countable union of the $S'_\lambda$ in $S$ has  a stronger version, namely,
   the corresponding points $[T_{E,0}]\in \mathbb{P}(T_{\mathcal{A}_b,0})$ are Zariski dense
   in the projectivized bundle $\mathbb{P}(T_{\mathcal{A}/S})$. The projection $\pi_\lambda$ above is thus generic and the composition $\pi_\lambda\circ g_A$ is generically finite as is $g_A$ and up to shrinking $S'$ if necessary,  its graph deforms in a flat way over the space of  parameters (namely a Zariski open set of $\mathbb{P}(T_{\mathcal{A}/S})$).
   This is now finished because we first restrict to
    the Zariski dense open set $U$  of $\mathbb{P}(T_{\mathcal{A}_{S/B},0})$ where
    the rational map $\pi_\lambda\circ g_A$ is generically finite and its graph deforms in a flat way, and then there are finitely many generically finite covers of $U$
     parameterizing a factorization of the rational  map
   $\pi_\lambda\circ g_A$.
   As the  diagram (\ref{numerodiag}) shows that
   there is a factorization of $\pi_\lambda\circ g_A$ as
   $$\Sigma'_{\mathcal{A}_b}\stackrel{p_\lambda}{\rightarrow }\Sigma'_{\mathcal{A}_b,p_\lambda}\stackrel{g_B}{\rightarrow }G(k',g),$$
    we conclude that all the maps  $\Sigma'_{\mathcal{A}_b}\stackrel{p_{\lambda,\Sigma}}{\rightarrow }\Sigma'_{\mathcal{A}_b,p_\lambda}$ are, up to birational equivalence of the target, members of  finitely many families of generically finite dominant rational maps
    $\psi:\mathcal{A}_b\dashrightarrow Y_b$.
   \end{proof}
   As a corollary, we conclude using the density of the union of  the sets $S'_\lambda$ that
there is, up to replacing $S'$ by a a generically finite cover   of it,
 a family of $k'$-dimensional varieties $\Sigma''_{\mathcal{A}_{S'}}$, together with a dominant generically finite rational map
 \begin{eqnarray}
 \label{eqpfam}p: \Sigma'_{\mathcal{A}_{S'}}\dashrightarrow \Sigma''_{\mathcal{A}_{S'}}\end{eqnarray}

 identifying birationally to $p_{\lambda,\Sigma}:\Sigma'_{\mathcal{A}_{b}}\dashrightarrow
 \Sigma'_{\mathcal{A}_{b,p_\lambda}}$ along each $S'_\lambda$.

 We now finish the proof by contradiction. Assume that $k'=k$. Then $\Sigma'_{\mathcal{A}_b,p_\lambda}$ must be a component $\Gamma_b$ of $\Sigma_{\mathcal{B}_b}$. In particular it does not depend on the elliptic curve $E$. Restricting to a dense  Zariski open set
 $S''$ of $S'$ is necessary, we can assume that we have  desingularizations
 \begin{eqnarray}
 \label{eqpfam1}\widetilde{\Sigma'_{\mathcal{A}_{S'}}}\stackrel{\tilde{p}}{\dashrightarrow } \widetilde{\Sigma''_{\mathcal{A}_{S'}}}\end{eqnarray}
  with smooth fibers over $S''$.
  Let $\tilde{j}:\Sigma'_{\mathcal{A}_b}\rightarrow \mathcal{A}_b$ be the natural map, and consider the morphism
  $$\tilde{p}_*\circ \tilde{j}^*:{\rm Pic}^0(\mathcal{A}_b)\rightarrow {\rm Pic}^0(\widetilde{\Sigma''_{\mathcal{A}_{b}}})$$
  which is a group morphism defined at the general point of $S''$.
  This morphism is nonzero because when $b\in S''_\lambda$ for some $\lambda$, it is injective
   modulo torsion on ${\rm Pic}^0(\mathcal{B}_b)$ (which maps by the  pull-back $p_\lambda^*$ to $ {\rm Pic}^0(\mathcal{A}_b)$
   with finite kernel). Indeed, by the projection formula, denoting by
    $\tilde{j'}:\widetilde{\Sigma''_{\mathcal{A}_{b}}}\rightarrow B$ the natural map, we  have the equality of maps from ${\rm Pic}^0(\mathcal{B}_b)$ to ${\rm Pic}^0(\widetilde{\Sigma''_{\mathcal{A}_{b}}})$:
    $$(\tilde{p}_*\circ \tilde{j}^*)_{\mid {\rm Pic}^0(\mathcal{B}_b)}=
    ((\tilde{p}_{\lambda,\Sigma})_*\circ \tilde{j}^*)_{\mid {\rm Pic}^0(\mathcal{B}_b)}=
    ((\tilde{p}_{\lambda,\Sigma})_*\circ(\tilde{p}_{\lambda,\Sigma})^*\circ \tilde{j'}^*={\rm deg}\,p_{\lambda,\Sigma} \,\tilde{j'}^*.$$
    We note here that the morphism $\tilde{j'}^*:{\rm Pic}^0(\mathcal{B}_b)\rightarrow{\rm Pic}^0(\widetilde{\Sigma''_{\mathcal{A}_{b}}})$ has finite kernel because
    ${\rm dim}\,{\rm Im}\,\tilde{j'}=k>0$.
      As the abelian variety
${\rm Pic}^0(\mathcal{A}_b)$ is  simple at the very general point of $S''$, the nonzero morphism
$(\tilde{p}_{\lambda,\Sigma})_*\circ \tilde{j}^*$ must be  injective. But then, by specializing at a point $b$ of $S''_\lambda$, where $\lambda$ is chosen in such a way that
 $S''_\lambda=S''\cap S'_\lambda$ is non-empty, we find that this morphism  is injective on
the component ${\rm Pic}^0(E_b)$ of ${\rm Pic}^0(\mathcal{A}_b)$. We can now fix the abelian variety  $\mathcal{B}_b$ and deform the elliptic curve $E_b$.
We then get  a contradiction, because we know that the variety $\widetilde{\Sigma''_{\mathcal{A}_{b}}}$ depends (at least birationally)  only on $\mathcal{B}_b$ and not
 on $E_b$, so that its Picard variety cannot contain
a variable elliptic curve $E_b$.
\end{proof}

\section{Proof of Theorems \ref{theochowring} and \ref{coroorbit}, (iv)\label{secproofcoriv}}
\subsection{\label{secdimestimate}Dimension estimate}
Recall that for an abelian variety $A$ and a nonnegative integer $k$, we denote by
$A_k\subset A$ the set of points $x\in A$ such that
$(\{x\}-\{0_A\})^{*k}=0$ in ${\rm CH}_0(A)$. The following proves item (i) of Theorem \ref{theochowring}:
\begin{prop}\label{proestimatechowring} For $k>0$, one has ${\rm dim}\, A_k\leq k-1$.
\end{prop}
\begin{proof} Let $g:={\rm dim}\,A$ and let $\Gamma^{Pont}_k$ be the codimension $g$ cycle of
$A\times A$ such that
 $$(\Gamma^{Pont}_{k })_*(x)=(\{x\}-\{0_A\})^{*k}.$$

for any $x\in A$. As $(\{x\}-\{0_A\})^{*k}=\sum_{i=0}^k(-1)^{k-i}\binom{k}{i}\{ix\}$, we can take
\begin{eqnarray}\label{eqstarkx}\Gamma^{Pont}_k=\sum_{i=0}^k(-1)^{k-i}\binom{k}{i}\Gamma_{i},
\end{eqnarray}
where $\Gamma_i\subset A\times A$  is the graph of the map $m_i$ of multiplication by $i$.
  Let us compute $(\Gamma_k^{Pont})^*\eta$ for any holomorphic form on $A$.
\begin{lemm} \label{lekform}
One has  $(\Gamma^{Pont}_k)^*\eta=0$ for any holomorphic form
$\eta$ of  degree $<k$ on $A$, and
 $(\Gamma_k^{Pont})^*\eta=k!\eta$ for a holomorphic form of degree $k$ on $A$.
\end{lemm}
\begin{proof} Indeed $m_i^*\eta=i^d\eta$, where $d={\rm deg}\,\eta$.
By (\ref{eqstarkx}), the lemma is thus equivalent to

(i) $\sum_{i=0}^k(-1)^{k-i}\binom{k}{i} i^d=0,\,d<k$,

(ii) $\sum_{i=0}^k(-1)^{k-i}\binom{k}{i} i^k=k!$.

From the formula  $\sum_{i=0}^k(-1)^{k-i}\binom{k}{i}X^i=(X-1)^k$, we get that the $d$-th derivative of the polynomial
$\sum_{i=0}^k(-1)^{k-i}\binom{k}{i}X^i$ at $1$ is $0$ for $d<k$, and is equal to $k!$ for $d=k$. This immediately implies (i)
by induction on $d$, using the fact that
$i(i-1)\ldots(i-d+1)-i^d$ is a degree $d-1$ polynomial in
$i$,  and then (ii) by the same argument.
\end{proof}
This lemma directly  implies Proposition \ref{proestimatechowring}. Indeed,
by Mumford's theorem \cite{mumford}, one has $(\Gamma^{Pont}_k)^*\eta_{\mid A_k}=0$ for any holomorphic form
$\eta$ of positive degree, and in particular for any holomorphic  $k$-form. By Lemma \ref{lekform},
we conclude that, denoting by $A_{k,reg}\subset A_k$ the regular locus of $A_k$, $\eta_{\mid A_k}=0$ for any holomorphic form
$\eta$ of  degree $k$
on $A$, that is, ${\rm dim}\,A_k<k$.
\end{proof}

\subsection{\label{secnatdef} Proof of Theorem  \ref{theochowring}}
The following result is  almost obvious:
\begin{lemm} \label{leAknat} For any integer $k\geq 0$, the set $A_k\subset A$ defined in (\ref{eqAk}) is naturally defined
in the sense of Definition \ref{definat}.
\end{lemm}
\begin{proof} It is known that the set $A_k\subset A$ is a countable union of closed algebraic subsets.
Using the fact that for a morphism $f:A\rightarrow B$ of abelian varieties,
$$f_*:{\rm CH}_0(A)\rightarrow {\rm CH}_0(B)$$
is compatible with the Pontryagin product, we conclude that
$f_*(A_k)\subset B_k$. Finally, given a family
$\pi:\mathcal{A}\rightarrow S$ of abelian varieties, the set
of points $x\in \mathcal{A}$ such that $(\{x\}-\{0_{\mathcal{A}_b}\})^{*k}=0$ in
${\rm CH}_0(\mathcal{A}_b)$, $b=\pi(x)$, is a countable union of closed algebraic subsets of
$\mathcal{A}$ whose fiber over $b\in S$ coincides set-theoretically with $\mathcal{A}_{b,k}$.
\end{proof}
\begin{proof}[Proof of Theorem \ref{theochowring}] The theorem follows from Proposition  \ref{proestimatechowring}, Lemma
\ref{leAknat}, and Theorem \ref{thegenpirola}.
\end{proof}

\subsection{Proof of Theorem \ref{coroorbit}, (iv)\label{sectheoorbitoA}}
We first prove the following Proposition (cf. Proposition \ref{theorelation}).
\begin{prop}\label{theorelation2} Let $A$ be an abelian variety and let $x_1,\ldots,\,x_k\in A$ such that
\begin{eqnarray}\label{eqhypo}\sum_i\{x_i\}=k\{0_A\}\,\,{\rm  in }\,\,{\rm CH}_0(A).
\end{eqnarray} Then
\begin{eqnarray}\label{eqpoutheorrelation} (\{x_i\}-\{0_A\})^{*k}=0 \,\,{\rm in}\,\,{\rm CH}_0(A)
\end{eqnarray}
for $i=1,\ldots,\,k$.

\end{prop}

\begin{proof}
Let $\gamma_l:=\sum_{|I|=l, I\subset \{2,\ldots,k\}}\{x_I\}$, where
$x_I:=\sum_{i\in I}x_i$. Then by (\ref{eqhypo}), we have \begin{eqnarray}\label{eqhypoprime}\gamma_1=\sum_{i=2}^{k}\{x_i\}=-\{x_1\}+k\{0_A\}.
\end{eqnarray} Furthermore,
$\gamma_l=0$ for $l\geq k$ and
the following inductive relation is obvious:
\begin{eqnarray}\label{eqpoutheorrelation1}(\sum_{i=2}^{k}\{x_i\})*\gamma_l=
(l+1)\gamma_{l+1}+((m_2)_*\gamma_1)*\gamma_{l-1}-((m_3)_*\gamma_1)*\gamma_{l-2}+\ldots,
\end{eqnarray}
that is:

\begin{eqnarray}\label{eqpoutheorrelation2} (l+1)\gamma_{l+1}=\sum_{i=0}^{l}(-1)^i((m_{i+1})_*\gamma_1)*\gamma_{l-i},
\end{eqnarray}
where by (\ref{eqhypoprime}), $(m_{i+1})_*\gamma_1=-\{(i+1)x_1\}+k\{0_A\}$.
Formula (\ref{eqpoutheorrelation2}) implies inductively  that
\begin{eqnarray}\label{eqpoutheorrelation3}\gamma_{l}=\sum_{i=0}^l\alpha_{l,i} \{i x_1\}\end{eqnarray} for some rational  nonzero coefficients $\alpha_{l,i}$. As
$$(\{x_1\}-\{0_A\})^{*i}=\sum_{j=0}^i(-1)^{i-j}\binom{i}{j} \{jx_1\},$$
 the $0$-cycles $\{jx_1\},\,0\leq j\leq l$ and $(\{x_1\}-\{0_A\})^{*j},\,0\leq j\leq l$
generate the same subgroup of ${\rm CH}_0(A)$.
The relation
$\gamma_k=0$ thus provides a nontrivial  degree $k$ linear relation  with $\mathbb{Q}$-coefficients between the
$0$-cycles
$$\{0_A\},\,\{x_1\}-\{0_A\},\,(\{x_1\}-\{0_A\})^{*k},$$
or equivalently a polynomial relation in the variable $\{x_1\}-\{0_A\}$ for the Pontryagin product, where
the scalars are mapped to $\mathbb{Q}\{0_A\}$.
As we know by \cite{bloch} that $(\{x_1\}-\{0_A\})^{*g+1}=0$, we conclude that
$(\{x_1\}-\{0_A\})^{*k}=0$.
\end{proof}
 Proposition \ref{theorelation2} says that if $\{x_1\}+\ldots+\{x_k\}=k\{0_A\}$ in ${\rm CH}_0(A)$, then
$x_i\in A_k$. The locus swept-out by the orbit $|k\{0_A\}|$ is thus contained in $A_k$. We thus deduce
 from Theorem \ref{theochowring} the following corollary:
 \begin{coro} (Cf. Theorem \ref{coroorbit}, (iv)) For any abelian variety $A$, the locus swept-out by the orbit $|k\{0_A\}|$
 has dimension $\leq k-1$. For   a very general abelian variety  $A$ of dimension $g\geq 2k-1$,  the orbit $|k\{0_A\}|$ is countable.
 \end{coro}
 In this statement, the locus swept-out by the orbit $|k\{0_A\}|$ is the set of points $x\in A$ such that
 a cycle $x+Z'$ with $Z'$ effective of degree $k-1$ belongs to $|k\{0_A\}|$. The dimension of this locus can be much smaller than the dimension of the orbit itself, as shown by the examples of orbits contained in subvarieties $C^{(k)}\subset A^{(k)}$ for some curve $C$.

 \section{Proof of Theorem \ref{coroorbit}, (i)\label{secdimorbit}}
We give in this section  the proof of item (i) in  Theorem \ref{coroorbit}. We first recall the statement:
\begin{theo}\label{ledimorbit}  Let $A$ be an abelian variety. The dimension of any orbit
$|Z|\subset A^{(k)}$ for rational equivalence is at most $k-1$.
\end{theo}
\begin{proof}
We will rather work with the inverse image $\widetilde{|Z|}$ of the orbit $|Z|$ in $A^k$. By Mumford's theorem \cite{mumford},
for any holomorphic $i$-form $\alpha$ on $A$ with $i>0$, one has, along the regular locus $\widetilde{|Z|}_{reg}$ of
$\widetilde{|Z|}$:
\begin{eqnarray}
\label{eqmumford} \sum_{j=1^k}pr_j^*\alpha_{\mid \widetilde{|Z|}_{reg}}=0,
\end{eqnarray}
where the $pr_j:A^k\rightarrow A$ are the various  projections.
 Let
$x=(x_1,\ldots,x_k)\in \widetilde{|Z|}_{reg}$ and let $V:=T_{\widetilde{|Z|}_{reg},x}\subset W^k$, where
$W=T_{A,x}=T_{A,0_A}$. One has ${\rm dim}\,V={\rm dim}\,|Z|$,
and (\ref{eqmumford}) says that:

\vspace{0.5cm}

(*) {\it  for any $\alpha\in \bigwedge^iW^*$ with $i>0$, one has
$(\sum_jpr_j^*\alpha)_{\mid V}=0$.}

\vspace{0.5cm}

Theorem \ref{ledimorbit} thus follows from the following proposition \ref{ledimorbitinfinitesimal}.
\end{proof}
\begin{prop}\label{ledimorbitinfinitesimal} Let $W$ be a vector space, $V\subset W^k$ be a vector subspace
satisfying property (*). Then ${\rm dim }\,V\leq k-1$.
\end{prop}
\begin{rema}{\rm If ${\rm dim}\,W=1$, the result is obvious, as $V\subset W^k_0\subset W$, where
$W^k_0:={\rm Ker}\,(\sigma: W^k\rightarrow W)$,  $\sigma$ being the sum map. If ${\rm dim}\,W=2$,
the result follows from the fact that, choosing a generator $\eta$ of
$\bigwedge^2W^*$, the $2$-form $\sum_jpr_j^*\eta$ is nondegenerate on $W^k_0$ (which has dimension $2k-2$). A subspace $V$ satisfying (*) is
contained in $W^k_0$ and totally isotropic for this $2$-form, hence has dimension $r\leq k-1$.}
\end{rema}
\begin{proof}[Proof of Proposition \ref{ledimorbitinfinitesimal}] Note that the group ${\rm Aut}\,W$ acts on $W^k$, with induced action on
 ${\rm Grass}(r,W^k)$ preserving the set of $r$-dimensional vector subspaces $V\subset W^k$ satisfying condition (*). Choosing a $\mathbb{C}^*$-action on ${W}$ with
finitely many fixed points $e_1,\ldots, e_n$, $n={\rm dim}\,W$,
the fixed points $[V]\in {\rm Grass}\,(r,W^k)$ under the induced action of  $\mathbb{C}^*$ on the Grassmannian are
of the form $V=\langle A_1e_1,\ldots A_n e_n\rangle$, where $A_i\subset (\mathbb{C}^k)^*$ are vector subspaces , with $r=\sum_i{\rm dim}\,A_i$. It suffices to prove the inequality $r\leq k-1$ at such a fixed point, which we do now.
The spaces $A_i$ have to satisfy the following conditions:

\vspace{0.5cm}

(**) {\it For any
$\emptyset\not=I=\{i_1,\ldots,\,i_s\}\subset \{1,\ldots,n\}$ and for any choices of $\lambda_l\in A_{i_l}$, $l=1,\ldots,\,s$,}
$$\sum_{j=1}^k(\lambda_1\ldots\lambda_s)(f_j)=0,$$
{\it where $f_j$ is the natural basis of $\mathbb{C}^k$. }

\vspace{0.5cm}

A better way to phrase condition (**) is to use the (standard) pairing $\langle\,,\,\rangle$ on $(\mathbb{C}^k)^*$, given
by
$$ \langle \alpha,\beta\rangle=\sum_{j=1}^k\alpha(f_j)\beta(f_j).$$
Condition (**) when there are only two nonzero spaces $A_i$ is
the following

\begin{eqnarray}\label{eqdeux1} \langle \alpha,\beta\rangle=0\,\,\forall \alpha\in A_1,\,\beta\in A_2\\
\label{equn}
\langle \alpha,e\rangle=0=\langle e,\beta\rangle=0 \,\,\forall \alpha\in A_1,\,\beta\in A_2,
\end{eqnarray}
where $e$ is the vector $(1,\ldots,1)\in (\mathbb{C}^k)^*$. Indeed, the case $s=2$ in (**)
provides (\ref{eqdeux1}) and the case $s=1$ in (**) provides (\ref{equn}).
The fact that the pairing $\langle\,,\,\rangle$ is nondegenerate on $(\mathbb{C}^k)^*_0:=e^{\perp}$ immediately implies that
$\sum_i{\rm dim}\,A_i\leq k-1$ when only two of the spaces $A_i$ are nonzero.
By the above arguments, the proof of Proposition \ref{ledimorbitinfinitesimal} is finished used the following lemma:
\begin{lemm}\label{lepourproduitdeformes} Let $A_i\subset (\mathbb{C}^k)^*,\,i=1,\ldots,\,n$, be linear subspaces satisfying conditions (**). Then
$\sum_i{\rm dim}\,A_i\leq k-1$.
\end{lemm}
\begin{proof} We will use the following result:
\begin{lemm}\label{ledeuxA} Let $A\subset \mathbb{C}^k,\,B\subset \mathbb{C}^k$ be vector subspaces
satisfying the following conditions:

(i) For any $a=(a_i)\in A,\,b=(b_i)\in B,\,\,\sum_ia_ib_i=0$.

(ii) For any $a=(a_i)\in A,\,b=(b_i)\in B,\,\,\sum_ia_i=0,\,\sum_ib_i=0$.

Then ${\rm dim}\,(A\cdot B +A+B)\geq {\rm dim}\,A+{\rm dim}\,B$, where
$A\cdot B$ is the vector subspace of  $\mathbb{C}^k$ generated by the elements $(a_ib_i)$, $a=(a_i)\in A,\,b=(b_i)\in B$.

\end{lemm}
Let us first show how Lemma \ref{ledeuxA} implies Lemma \ref{lepourproduitdeformes}. Indeed, we can argue inductively on the number $n$ of spaces $A_i$. As already noticed, Lemma \ref{lepourproduitdeformes} is easy when $n=2$. Assuming the statement is proved for $n-1$,  let $A_1,\ldots,\,A_n$ be as in Lemma \ref{lepourproduitdeformes}  and let  $A'_1=A_1,\,\ldots,\,A'_{n-2}=A_{n-2}$
and $A'_{n-1}=A_{n-1}\cdot A_n+A_{n-1}+A_n$.
Then the set of spaces $A'_1,\ldots,\,A'_{n-1}$ satisfies conditions (**), and on the other hand
Lemma \ref{ledeuxA} applies to the pair $(A,B)=(A_{n-1},A_n)$ as they satisfy the desired conditions by
(**). Hence we have ${\rm dim}\,A'_{n-1}\geq {\rm dim}\,A_{n-1}+ {\rm dim}\,A_{n}$ and
by induction on $n$, $\sum_{i=1}^{n-2}{\rm dim}\,A'_{i}+{\rm dim}\,A'_{n-1}\leq k-1$.
Hence $\sum_{i=1}^{n}{\rm dim}\,A_{i}\leq k-1$.
\end{proof}
\begin{proof}[Proof of Lemma \ref{ledeuxA}] Let $A_1:= e+A,\,B_1=e+B\subset \mathbb{C}^k$.
Under the conditions (i) and (ii), the multiplication map
$$\mu: A_1\times B_1\rightarrow \mathbb{C}^k$$
$$\mu((a_i),(b_i))=(a_ib_i)$$
has image in the affine space $\mathbb{C}^k_1:=e+\mathbb{C}^k_0$, where
$\mathbb{C}^k_0=e^{\perp}$, and more precisely it generates the affine space
$e+A+B+A\cdot B\subset e+\mathbb{C}^k_0$.
It thus  suffices to show that
the dimension of the algebraic set ${\rm Im}\,\mu$ is at least ${\rm dim}\,A+{\rm dim}\,B$. Lemma \ref{ledeuxA} is thus implied by the following:

\begin{claim}\label{claim} The map $\mu$ has finite fiber  near the point $(e,e)\in A_1\times B_1$.
\end{claim}
The proof of the claim is as follows: Suppose $\mu$ has  a positive dimensional fiber passing through
$(e,e)$. We choose an irreducible curve  contained in the fiber, passing through $(e,e)$ and
with normalization $C$. The curve $C$ admits rational functions $\sigma_i,\,i=1,\ldots,\,k$ mapping it to
$A_1$ such that the functions $\frac{1}{\sigma_i}$ map $C$ to $B_1$.
The conditions (i) and (ii) say that $$\sum_i\sigma'_i(s)\frac{1}{\sigma_i(t)}=0$$ as a function of $(s,t)$ for any choice of points $x,\,y\in C$ and local  coordinates $s,\,t$ near $x$, resp. $y$, on $C$. We now do $x=y$ and choose for $x$ a pole (or a zero) of one of the  $\sigma_l$'s. We assume that the local coordinate $s$ is centered at $x$, and  write
$\sigma_i(s)=s^{d_i}f_i$, with $f_i$ a holomorphic function of $s$ which is  nonzero at $0$.
We then get
\begin{eqnarray}\label{eqpourfonction} \sigma'_i(s)\frac{1}{\sigma_i(t)}=d_i \frac{s^{d_i-1}}{t^{d_i}} \phi_i(s,t)+\frac{s^{d_i}}{t^{d_i}}\psi_i(s,t),
\end{eqnarray}
where $\phi_i(s,t)$ is holomorphic in $s,\,t$ and takes value $1$ at $(x,x)=(0,0)$ and $\psi_i(s,t)$ is holomorphic in $s,\,t$.
 Restricting to a curve
$D\subset C\times C$ defined by the equation $s=t^l$ for some chosen $l\geq2$, the function
$(\sigma'_i(s)\frac{1}{\sigma_i(t)})_{\mid D}$ has order $l(d_i-1)-d_i=(l-1)d_i-l$  and first nonzero coefficient in its Laurent development equal to $d_i$. These orders are different for distinct $d_i$ and
 the vanishing
 $\sum_i\sigma'_i(s)\frac{1}{\sigma_i}(t)=0$ is then clearly impossible: indeed, by pole order considerations, for the minimal negative value $d$ of $d_i$, hence minimal value of  the numbers $(l-1)d_i-l$, the
 first nonzero coefficient in the Laurent development of $(\sigma'_i(s)\frac{1}{\sigma_i(t)})_{\mid D}$ should be also $0$ and it is the same as for the  sum
 $\sum_{i,\,d_i=d}(\sigma'_i(s)\frac{1}{\sigma_i}(t))_{\mid D}$, which   is  equal to $M_d d$, where $M_d$ is the cardinality of the set
 $\{i,\,d_i=d\}$.

The claim is proved.
\end{proof}
The proof of Proposition \ref{ledimorbitinfinitesimal} is thus finished.
\end{proof}
\subsection{An alternative proof of Theorem \ref{coroorbit}, (iv)}
As a first application, let us give a second proof of Theorem \ref{coroorbit}, (iv).
The general dimension
estimate of Theorem \ref{coroorbit}, (i) implies  that the locus swept-out by the orbit of $|k0_A|$ is of dimension
$\leq k-1$ for any abelian variety $A$. This locus is clearly  naturally defined. Hence by Theorem
\ref{thegenpirola}, (ii), it is countable for a very general abelian variety
of dimension $\geq 2k-1$.

\section{Proof of Theorem \ref{coroorbit}, (ii) and (iii)}
Theorem \ref{coroorbit}, (iv) has been proved in Section \ref{sectheoorbitoA}.
We  will now prove the following result by induction on $l\in\{0,\ldots,k\}$:
\begin{prop} \label{proinduction} For $g\geq 2^{l}(2k-1)+(2^{l}-1)(k-2)$, and $A$ a very general abelian variety
of dimension $g$, any $0$-cycle of the form $(k-l)\{0_A\}+Z$, with $Z\in A^{(l)}$, has countable orbit.
\end{prop}

The case $l=0$  is Theorem \ref{coroorbit}, (iv) and the case $l=k$ is then Theorem \ref{coroorbit}, (ii). The case $l=k-2$ is
Theorem \ref{coroorbit}, (iii).

It thus only remains  to prove Proposition \ref{proinduction}.
For clarity, let us write-up the detail of the first induction step:
Let $\Sigma_1(A)\subset A$ be the set
of points $x\in A$ such that the   orbit    $|(k-1)\{0_A\}+\{x\}|\subset A^{(k)}$ is positive dimensional.
The set $\Sigma_1(A)$ is a countable union of closed algebraic subsets of $A$.
We would like to show that $\Sigma_1(A)$ is naturally defined in the sense of Definition
\ref{definat}, and there is a small difficulty here: suppose that $p:A\rightarrow B$ is a morphism of abelian
varieties, and let $|Z|\subset A^{(k)} $ be a positive dimensional orbit for rational equivalence on $A$. Then
$p_*(|Z|)\subset B^{(k)} $ could be zero-dimensional.
In the case where $Z=(k-1)\{0_A\}+\{x\}$, this prevents a priori proving that $\Sigma_1(A)$ satisfies
 axiom (ii) of Definition \ref{definat}.
This problem can be circumvented using the following lemma which has been in fact already used  in the proof of Theorem \ref{thegenpirola}.
 Let $\mathcal{A}\rightarrow S$ be a generically complete family of abelian varieties of dimension $g$. This means that we fixed a polarization type $\lambda$ and the moduli map
 $S\rightarrow\mathcal{A}_{g,\lambda}$ is dominant.

 \begin{lemm}\label{generic} Let   $\mathcal{W}\subset \mathcal{A}$ be a closed algebraic subset which is
 flat over $S$ of relative dimension $k'$. Then:

  (i) For any $b\in S$, any morphism $p: \mathcal{A}_b\rightarrow B$ of abelian varieties
  with ${\rm dim}\,B\geq k'$,  $p(\mathcal{W}_b)\subset B$ has dimension $k'$.

  (ii) Assume $k'>0$. For any $b\in S$, any morphism $p: \mathcal{A}_b\rightarrow B$ of abelian varieties with ${\rm dim}\,B>0$,  $p(\mathcal{W}_b)\subset B$ has positive dimension.
\end{lemm}
\begin{proof} (i) Indeed, the locally constant class $[\mathcal{W}_b]\in H^{2g-2k'}(\mathcal{A}_b,\mathbb{Q})$ must be a nonzero multiple of $\theta_\lambda^{g-k'}$, since for very generic
$b\in S$, these are the only nonzero Hodge classes on $\mathcal{A}_b$. We thus  have, using our assumption that
${\rm dim}\,B\geq k'$,
$p_*([\mathcal{W}_b])\not=0$ in $H^{2{\rm dim}\,B-2k'}(B,\mathbb{Q})$, which implies
that   ${\rm dim}\,p(\mathcal{W}_b)=k'$.

Statement (ii) is obtained as an application of  (i) in the case $k'=1$. One first reduces to this case by taking complete intersection curves in $\mathcal{W}_b$ in order to reduce to the case $k'=1$.
\end{proof}
In the following corollary,  the orbits for rational equivalence of $0$-cycles of $X$ are  taken in $X^l$ rather than $X^{(l)}$.
\begin{coro} \label{coropourtheoorbitgen} The situation being as in Lemma \ref{generic},
let $\mathcal{W}\subset \mathcal{A}^{l/S}$ be a family of positive dimensional orbits for rational equivalence in the fibers. Then, up to shrinking $S$ if necessary, for any $b\in S$, any morphism $p: \mathcal{A}_b\rightarrow B$ of abelian varieties, where $B$ is an abelian variety of dimension $>0$, and any $i=1,\ldots,l$, $p^l(\mathcal{W}_b)$ is a positive dimensional  orbit of $B$.
\end{coro}
\begin{proof}  Indeed, by specialization, $\mathcal{W}_b $ is a positive dimensional
 orbit for rational equivalence in $\mathcal{A}_b^{l}$. Up to shrinking
$S$, we can assume that  the restrictions  $\pi_{\mid  pr_i(\mathcal{W})}: pr_1(\mathcal{W})\rightarrow S$ are flat for all $i$.
Our assumption is that for one $i$, $pr_i(\mathcal{W})$ has positive relative dimension over $S$.
Lemma \ref{generic}, (ii),  then implies that $pr_i(p^l(\mathcal{W}_b))$ has positive
dimension, so that $p^l(\mathcal{W}_b)$ is a positive dimensional orbit for rational equivalence of $0$-cycles of $B$.
\end{proof}

\begin{proof}[Proof of Proposition \ref{proinduction}] Let now $A$ be a very general abelian variety. This means that for some generically complete family $\pi: \mathcal{A}\rightarrow S$ of polarized abelian varieties, $A$ is isomorphic to the fiber over a very general point of $S$. As $A$ is very general,  the locus $\Sigma_1(A)$  is the specialization of  the corresponding
 locus $\Sigma_1(\mathcal{A}/S)$ of $\mathcal{A}$, and more precisely, of the union of its  components dominating $S$.
 For any fiber $\mathcal{A}_b$, let us define the deformable locus $\Sigma_1(A)_{def}$
  as the one which is obtained by specializing to $\mathcal{A}_b$ the union of the dominating components of the locus of the   relative locus  $\Sigma_1(\mathcal{A}/S)$. For a very general abelian variety $A$,
  $\Sigma_1(A)=\Sigma_1(A)_{def}$ by definition.
 Corollary \ref{coropourtheoorbitgen} essentially says that this locus is naturally defined. This is not quite true because the definition of $\Sigma_1(A)_{def}$ depends on choosing a family $\mathcal{A}$ of deformations of $A$ (that is, a polarization on $A$). In the axioms of Definition \ref{definat}, we thus should work, not with abelian varieties but with polarized  abelian varieties. Axiom (i) should
 be replaced by its family version, where $\mathcal{A}\rightarrow S$ is locally complete,
 $S'\subset S$ is a subvariety, $f:\mathcal{A}_{S'}\rightarrow \mathcal{B}$ is a morphism of abelian varieties over
 $S'$, and $\mathcal{B}\rightarrow {S'}$ is locally complete.
 We leave to the reader proving that Theorem  \ref{thegenpirola} extends to this context.

Assume now $g\geq 2k-1$. Then $\Sigma_1(A)$, hence a fortiori $\Sigma_1(A)_{def}$, is different from
$A$. Indeed, otherwise, for any $x\in A$, $(k-1)\{0_A\}+\{x\}$ has positive dimensional orbit, hence taking
$x=0_A$, we get that $k\{0_A\}$ has positive dimensional orbit, contradicting Theorem \ref{coroorbit}, (iv). Theorem
\ref{thegenpirola}, (i) then implies that for $g\geq 2(2k-1)-1$, $\Sigma_1(A)_{def}$ is countable. Hence there are
only countably many positive dimensional orbits of the form $|(k-1)\{0_A\}+\{x\}|$ and the locus they sweep-out forms by Corollary \ref{coropourtheoorbitgen} a naturally defined locus in $A$, which is of dimension $\leq k-1$ by Theorem  \ref{ledimorbit}.
It follows by applying Theorem \ref{thegenpirola}, (iii), that for $g\geq 2(2k-1)+k-2$, this locus itself is  countable, that is, all the orbits  $|(k-1)\{0_A\}+\{x\}|$
are countable for $A$ very general.

The general induction step works exactly in the same way, introducing the locus
$\Sigma_l(A)\subset A$ of points $x_l\in A$ such that
$(k-l)0_A+x_1+\ldots+x_l$ has a positive dimensional orbit for rational equivalence in $A$ for some points $x_1,\ldots,\,x_{l-1}\in A$.
\end{proof}

\section{Further discussion \label{secdiscussion}}

It would be nice to improve the estimates in our main theorems. As already mentioned  in the introduction, none of them  seems to be optimal.
Let us introduce a  naturally defined locus (or the deformation variant of that notion
 used in the last section) whose study should lead to
a proof of Conjecture \ref{conj}.
\begin{Defi} The locus $Z_A\subset A$ of positive dimensional normalized orbits of degree $k$
is the set of points $x\in A$ such that for some degree $k$
zero-cycle $Z=x+Z'$, with $Z'$ effective, one has
$${\rm dim}\,|Z|>0,\,\,\sigma(Z)=0.$$
\end{Defi}
Here $\sigma: A^{(k)}\rightarrow A$ is the sum map. It is constant along orbits under rational
equivalence. This locus, or rather its deformation version, is naturally defined.
Note also that by definition it is either of positive dimension or empty.
The main remaining  question is to estimate the dimension of this locus, at least
for very general abelian varieties. Conjecture
\ref{conj} would follow from:
\begin{conj}\label{conjchow} If $A$ is a very general abelian variety, the locus $Z_A\subset A$ of positive dimensional normalized orbits of degree $k$ has dimension $\leq k-1$.
\end{conj}
Conjecture \ref{conjchow} is true for $k=2$. Indeed, in this case the normalization condition reads $Z=\{x\}+\{-x\}$ for some $x\in A$. The positive dimensional normalized orbits of degree $2$ are thus also positive dimensional orbits of points in the
Kummer variety $K(A)=A/\pm Id$ of $A$. These orbits are rigid because on a surface
in $K(A)$
swept-out by a continuous family of such orbits, any holomorphic $2$-form on $K(A)$ should vanish
while $\Omega^2_{K(A)_{reg}}$ is generated by its sections.

It would be tempting to try to estimate the dimension of the locus of  positive dimensional normalized orbits of degree $k$ for any abelian
variety. Unfortunately, the following example shows that
this locus can be the whole of $A$:
\begin{example}{\rm Let $A$ be an abelian variety which has a degree $k-1$
positive dimensional
orbit $Z\subset A^{(k-1})$. Then for each $x\in A$, $\{x_1+x\}+\ldots+\{x_{k-1}+x\},\,\{x_1\}+\ldots+\{x_{k-1}\}\in Z$ is a positive dimensional orbit
and thus
the set $\{\{x_1+x\}+\ldots+\{x_{k-1}+x\}+\{-\sum_ix_i-(k-1)x\}$ is a positive dimensional normalized orbit of degree $k$.
In this case, the
locus  of positive dimensional normalized orbits of degree $k$ of $A$ is the whole of $A$.
}
\end{example}
Nevertheless, we can observe the following small evidence for Conjecture \ref{conjchow}:
\begin{lemm} \label{ledimorbitavechypot} Let $O\subset A^k$ be a closed  irreducible algebraic subset which is a union   of positive dimensional normalized orbits of degree $k$. Let $Z\in O_{reg}$ and assume
the positive dimensional orbit $O_Z$ passing through $Z$ has a tangent vector
$(u_1,\ldots, u_k)$ such that the vector space $\langle u_1,\ldots,u_k\rangle\subset T_{A,0_A}$
is of dimension $k-1$. Then the locus swept-out by the $pr_i(O)\subset A$ has dimension $\leq k-1$.
\end{lemm}
Note that $k-1$ is the maximal possible dimension of  the vector space $\langle u_1,\ldots,u_k\rangle$ because $\sum_iu_i=0$. The example above is a case where the vector space
$\langle u_1,\ldots,u_k\rangle$ has dimension $1$.

Applying Theorem \ref{thegenpirola}, (ii), Conjecture \ref{conjchow} in fact  implies
the following
\begin{conj}\label{conjchowdeux} If $A$ is a very general abelian variety of dimension $\geq 2k-1$, the locus  of positive dimensional normalized orbits of degree $k$ of $A$ is empty.
\end{conj}
 This is a generalization of  Conjecture \ref{conj}, because a $k$-gonal
curve $\tilde{j}:\widetilde{C}\rightarrow A,\,D\in W_k^1(C)$ can always be translated in such a way
that $\sigma(\tilde{j}_*D)=0$, hence becomes  contained in  the locus  of positive dimensional normalized orbits of degree $k$ of $A$.

We discussed in this paper only the  applications to  gonality. The case of higher dimensional linea systems
would be also interesting to investigate.
In a similar but different vein, the following problem is intriguing:

\begin{question}\label{questionplancurve} Let $A$ be a very general abelian variety. Is it true that
there is no curve $C\subset A$, whose normalization is a smooth plane curve?
\end{question}

If the answer to the above question is affirmative, then one could get examples of surfaces of general type which are not birational to a normal surface in $\mathbb{P}^3$. Indeed, take a
surface whose Albanese variety
is a general abelian variety as above. If $S$ is birational to a normal surface $S'$ in $\mathbb{P}^3$,
there are plenty of smooth plane curves in $S'$, which clearly map nontrivially
to ${\rm Alb}\,S$, which would be a contradiction.

Coll\`{e}ge de France, 3 rue d'Ulm, 75005 Paris FRANCE

claire.voisin@imj-prg.fr

\begin{thebibliography}{99}
\bibitem{alzatipirola}  A. Alzati, G. P. Pirola.  Rational orbits on three-symmetric products of abelian varieties. Trans. Amer. Math. Soc. 337 (1993), no. 2, 965-980.
\bibitem{lazarsfeld}  F. Bastianelli, P. De Poi, L. Ein, R. Lazarsfeld, B. Ullery. Measures of irrationality for hypersurfaces of large degree. Compos. Math. 153 (2017), no. 11, 2368-2393.
\bibitem{beauabelian}  A. Beauville. Sur l'anneau de Chow d'une vari\'et\'e ab\'elienne.  Math. Ann. 273 (1986), no. 4, 647-651.
    \bibitem{beauvillefourier} A. Beauville. Quelques remarques sur la transformation de Fourier dans l'anneau de Chow d'une vari\'et\'e ab\'elienne. in {\it Algebraic geometry} (Tokyo/Kyoto, 1982), 238-260,
Lecture Notes in Math., 1016, Springer, Berlin, 1983.
\bibitem{beaujac} A. Beauville. Algebraic cycles on Jacobian varieties.
Compos. Math. 140 (2004), no. 3, 683-688.
\bibitem{Beau} A. Beauville. Vari\'et\'es K\"ahleriennes dont la premi\`ere classe de Chern est nulle. J. Differential Geom. 18 (1983), no. 4, 755-782 (1984).
\bibitem{bloch} S. Bloch. Some elementary theorems about algebraic cycles on Abelian varieties. Invent. Math. 37 (1976), no. 3, 215-228.
\bibitem{colombo} E. Colombo, B. van Geemen. Note on curves in a Jacobian. Compositio Math. 88
(1993), no. 3, 333-353.
\bibitem{herbaut} F. Herbaut. Algebraic cycles on the Jacobian of a curve with a linear system of
given dimension. Compos. Math. 143 (2007), no. 4, 883-899.
\bibitem{huy} D. Huybrechts. Curves and cycles on K3 surfaces.  Algebr. Geom. 1 (2014), no. 1, 69-106.
    \bibitem{marian} A. Marian, Xiaolei Zhao. On the group of zero-cycles of holomorphic symplectic varieties,  arXiv:1711.10045.
\bibitem{mumford} D. Mumford. Rational equivalence of 0-cycles on surfaces, J. Math. Kyoto Univ. 9 (1969), 195-204.
\bibitem{pirola} G. P. Pirola. Curves on generic Kummer varieties, Duke Math. J. 59 (1989), 701-708.
\bibitem{pirola2} G. P. Pirola. Abel-Jacobi invariant and curves on generic abelian varieties,  in {\it Abelian varieties (Egloffstein, 1993)}, 237-249, de Gruyter, Berlin, (1995).
\bibitem{voisindiag} C. Voisin.  Some new results on modified diagonals. Geom. Topol. 19 (2015), no. 6, 3307-3343.
    \bibitem{voisink3} C. Voisin. Rational equivalence of 0-cycles on K3 surfaces and conjectures of Huybrechts and O'Grady in {\it Recent advances in algebraic geometry}, 422-436, London Math. Soc. Lecture Note Ser., 417, Cambridge Univ. Press, Cambridge, 2015.
    \bibitem{voisinisot} C. Voisin. Remarks and questions on coisotropic subvarieties and 0-cycles of hyper-K\"{a}hler varieties, in {\it  K3 surfaces and their moduli}, 365-399, Progr. Math., 315, Birkh\"auser/Springer,  (2016).
\bibitem{voisingodeaux} C. Voisin. Sur les z\'ero-cycles de certaines hypersurfaces munies d'un automorphisme.  Ann. Scuola Norm. Sup. Pisa Cl. Sci. (4) 19 (1992), no. 4, 473-492.
\end{thebibliography}
    \end{document}